\documentclass[12pt]{amsart}
\usepackage{amscd,amssymb}
\usepackage{url}
\usepackage{amscd,amssymb,latexsym,subfigure,hyperref}
\usepackage[graph,frame,poly,arc]{xy}  

\theoremstyle{plain}
\newtheorem{thm}[subsection]{Theorem}
\newtheorem*{problem}{Problem}

\newtheorem{prop}[subsection]{Proposition}
\newtheorem{cor}[subsection]{Corollary}
\theoremstyle{definition}
\newtheorem{rk}[subsection]{Remark}
\newtheorem{definition}[subsection]{Definition}
\newtheorem{ex}[subsection]{Example}

\newtheorem{question}[subsection]{Question}

\numberwithin{equation}{section}
\newcommand{\CC}{{\mathcal C}}

\newcommand{\al}{{\alpha}}
\newcommand{\be}{{\beta}}

\newcommand{\C}{\mathbb{C}}
\newcommand{\PP}{\mathbb{P}}

\begin{document}
\title [On plane conic arrangements with nodes and tacnodes]
{On plane conic arrangements with nodes and tacnodes}

\author[Alexandru Dimca]{Alexandru Dimca}
\address{Universit\'e C\^ ote d'Azur, CNRS, LJAD and INRIA, France and Simion Stoilow Institute of Mathematics,
P.O. Box 1-764, RO-014700 Bucharest, Romania}
\email{dimca@unice.fr}

\author[Marek Janasz]{Marek Janasz}
\address{Department of Mathematics,
Pedagogical University of Krakow,
Podchor\c a\.zych 2,
PL-30-084 Krak\'ow, Poland}
\email{marek.janasz@up.krakow.pl}

\author[Piotr Pokora]{Piotr Pokora}
\address{Department of Mathematics,
Pedagogical University of Krakow,
Podchor\c a\.zych 2,
PL-30-084 Krak\'ow, Poland}
\email{piotr.pokora@up.krakow.pl}

\subjclass[2010]{Primary 14N20; Secondary  14C20, 32S22}

\keywords{conic arrangements, nodes, tacnodes}

\begin{abstract} In the present paper, we study arrangements of smooth plane conics having only nodes and tacnodes as the singularities. We provide an interesting estimation on the number of nodes and tacnodes that depends only on a linear function of the number of conics. Based on that result, we obtain a new upper bound on the number of tacnodes which turns out to be better than Miyaoka's bound for a large enough number of conics. We also study the freeness and nearly freeness of such arrangements providing a detailed description.
\end{abstract}
 
\maketitle

\section{Introduction}
In the present paper we study the geometry of arrangements of smooth plane conics such that the only singular points are nodes ($A_{1}$ singularities) and tacnodes ($A_{3}$ singularities). These are the simplest (interesting) singularities that an arrangement of smooth conics can produce. It is worth pointing out that even for this setting, which looks restrictive, there is an extremely fundamental question regarding these types of singularities.
\begin{problem}
Let $\mathcal{C}$ be an arrangement of $k\geq 2$ smooth conics in $\mathbb{P}^{2}_{\mathbb{C}}$ which admits only nodes and tacnodes as singularities. What is the maximal possible number of tacnodes determined by $\mathcal{C}$?
\end{problem}
If $\mathcal{C} \subset \mathbb{P}^{2}_{\mathbb{C}}$ is an arrangement of $k\geq 2$ smooth conics, then by a simple count one has
\begin{equation}
\label{eq:comb}
  4\cdot \binom{k}{2} =  n + 2t,  
\end{equation}
where $n$ is the number of nodes and $t$ is the number of tacnodes. In particular, this gives
\begin{equation}
\label{eq:naive}
t \leq k(k-1),
\end{equation}
and this bound is, for small values of $k$, even sharp. For instance, we if we take $k=2$ we see that the maximal possible number of tacnodes is $2$, and this arrangement can be realized over the reals. The discrepancy occurs for $k=5$ when the above naive bound gives $t=20$, but an interesting result (analysis) by Megyesi \cite{Meg} gives that in reality we can have at most $17$ tacnodes. In fact, we have a complete classification of the arrangements which admit exactly $17$ tacondes, and one of these arrangements is the famous Naruki's pentagram arrangement \cite{Naruki}.

For any arrangement of $k\geq 3$ smooth plane conics, Miyaoka in \cite{Miyaoka} considered the double cover of the plane branched  along  the union of  the conics  and  after  applying the Miyaoka-Yau  inequality to it, or by taking  the boundary  divisor  $B$  consisting  of  the union of  conics  with  coefficient   $\frac{1}{2}$ and then applying the log-Miyaoka-Yau inequality to the pair  $(\mathbb{P}^{2}, B)$, he obtained the following upper-bound
\begin{equation}
\label{miyaoka}
t \leq \frac{4}{9}k^{2} + \frac{4}{3}k.
\end{equation}
Observe that this bound is sharp and much better than the naive bound, for instance if we take $k=5$, then
$$t \leq \frac{100}{9} + \frac{60}{9} \approx 17.7,$$
which means that the bound is sharp. 

However, if we restrict our attention to arrangements of smooth plane conics with ordinary singularities, i.e., all the intersection points look locally as $\{x^{t} = y^{t}\}$ for some $t\geq 2$, then an interesting result due to Tang \cite{Tang} provides the following inequality. For an arrangement $\mathcal{C}$ of $k$ smooth plane conics with only ordinary singularities, we denote by $t_{r}(\mathcal{C}) = t_{r}$ the number of $r$-fold intersection points of $\mathcal{C}$. 
\begin{thm}[Tang]
Let $\mathcal{C} \subset \mathbb{P}^{2}_{\mathbb{C}}$ be an arrangement $k\geq 3$ smooth conics with only ordinary singularities and such that there is no point where all conics meet simultaneously. Then
$$t_{2} + t_{3} + 5k \geq \sum_{r\geq 5}(r-4)t_{r}.$$
\end{thm}

In the present note we want to understand the geometry of arrangement of smooth plane conics in the plane that admit only nodes and tacnodes as singularities. Our main result provides an estimation on the difference between the number of nodes and tacnodes of such arrangements, and according to our knowledge this result is the first one in this direction.

\begin{thm}
\label{thm:nt}
Let $\mathcal{C} = \{C_{1}, ..., C_{k}\} \subset \mathbb{P}^{2}_{\mathbb{C}}$ be an arrangement of $k\geq 6$ smooth conics which admits only $n$ nodes and $t$ tacnodes. Then we have
$$t  \leq  \frac{1}{4}n + 5k.$$
\end{thm}
Applying (\ref{eq:comb}) to Theorem \ref{thm:nt} allows us to conclude the following upper bound.
\begin{cor}
Let $\mathcal{C} = \{C_{1}, ..., C_{k}\} \subset \mathbb{P}^{2}_{\mathbb{C}}$ be an arrangement of $k\geq 6$ smooth conics which admits only nodes and tacnodes as singularities. Then
\begin{equation}
\label{tac}
t \leq \frac{k^{2}}{3} + 3k.
\end{equation}
\end{cor}
\begin{rk}
Observe that for $k \geq 16$ our upper bound (\ref{tac}) is tighter than Miyaoka's upper bound (\ref{miyaoka}).
\end{rk}
In the second part of the note we focus on the problem of whether arrangements of smooth plane conics with nodes and tacnodes can deliver new examples of (nearly) free curves. It is well-known that if $\mathcal{L}$ is an arrangement of $k\geq 3$ lines having only double intersection points, then $\mathcal{L}$ is free if and only if $k=3$. We investigate the case of smooth conics and we provide a detailed answer, in Section 3, on this question by presenting the following results.

\begin{prop}
\label{prop1}
Let $\CC \subset \mathbb{P}^{2}_{\mathbb{C}}$ be an arrangement of $k \geq 2$ smooth conics with only nodes and tacnodes as singularities. Then $\CC$  is not free.
\end{prop}

\begin{prop}
\label{prop2}
Let $\CC \subset \mathbb{P}^{2}_{\mathbb{C}}$ be an arrangement of $k \geq 2$ smooth conics with only nodes and tacnodes as singularities. Then $\CC$ is nearly  free if and only if
$$ k \leq 4 \text{ and } t=k(k-1).$$
In the case when $\CC$ is nearly free, then the exponents are either $(k-1,k+1)$ or $(k,k)$.
\end{prop}

In the case of arrangements of higher degree curves an important role is played by the topological-analytical properties of singularities. In Section 4, we study quasi-homogeneous singularities which are ordinary and have an arbitrary multiplicity. This setting is interesting due to the fact that these two notions, in general, go apart rapidly.  We also refer to Section 4 for definitions related to singularities of plane curves (if the reader is not familiar with them).

\section{A combinatorial bound on conics with nodes and tacnodes}
Here we present all techniques and details regarding Theorem \ref{thm:nt}. As it was mentioned in the introduction, one way to obtain a reasonable upper bound on the number of tacnodes is to perform a double cover and use the Miyaoka-Yau inequality. We will somehow follow this path of research by using the logarithmic version of the Miyaoka-Yau inequality provided by Langer \cite{Langer}. This inequality seems to be more efficient and easier to use since we do not need to perform a complicated construction of the covering. On the other side, in order to apply the logarithmic Miyaoka-Yau inequality we need to know the local orbifold numbers of singularities. We do not aim to provide a detailed description here (due to a highly technical nature), but we roughly explain that if $p$ is a singular point, then the local orbifold Euler number, analytic in its nature, informs us about the order of the local orbifold fundamental group around this singularity. We are going to use the following version of Langer's result \cite[Section 11.1]{Langer}.
\begin{thm}[Langer]
Let $C \subset \mathbb{P}^{2}_{\mathbb{C}}$ be a reduced curve of degree $d$ and consider a log canonical pair $(\mathbb{P}^{2}_{\mathbb{C}},\alpha C)$ for a suitably chosen $\alpha$, then we have the following inequality:
\begin{equation}
\label{logMY}
\sum_{p \in {\rm Sing}(C)}  3\bigg( \alpha (\mu_{p} - 1) + 1 - e_{orb}(p,\mathbb{P}^{2}_{\mathbb{C}}, \alpha C) \bigg) \leq (3\alpha - \alpha^{2})d^{2} - 3\alpha d,
\end{equation}
where $\mu_{p}$ is the Milnor number of a singular point $p \in {\rm Sing}(C)$ and $e_{orb}(p,\mathbb{P}^{2}_{\mathbb{C}}, \alpha C)$ is the local orbifold Euler number of $p$.
\end{thm}
In the context of the above result, instead of providing technical details, let us present how the mentioned numbers look like for nodes and tacnodes.
It is classically known (and follows from definitions) that if $p$ is a node, then its Milnor number is equal to $1$, and if $q$ is a tacnode, then its Milnor number is equal to $3$. Now we need to know what are the local orbifold Euler numbers for these singularities. If $p$ is a node, then Langer proved that $e_{orb}(p,\mathbb{P}^{2}_{\mathbb{C}}, \alpha C)=(1-\alpha)^2$ provided that $0 \leq \alpha \leq 1$, and if $q$ is a tacnode, then $e_{orb}(q,\mathbb{P}^{2}_{\mathbb{C}}, \alpha C) =(1-2 \alpha)$ provided that $0 \leq \alpha \leq \frac{1}{4}$. For details regarding the presented orbifold Euler numbers, we refer to \cite[Section 8, 9]{Langer}.
We are ready to show our proof of Theorem \ref{thm:nt} that goes along the lines presented in \cite{Pokora}.
\begin{proof}
Let $\mathcal{C} = \{C_{1}, ..., C_{k}\} \subset \mathbb{P}^{2}_{\mathbb{C}}$ be an arrangement of $k\geq 6$ smooth conics and such that it admits only nodes and tacnodes. Let $C := C_{1} + ... + C_{k}$ so ${\rm deg} \, C = 2k$. In order to apply the orbifold Miyaoka-Yau inequality, our pair $(\mathbb{P}^{2}_{\mathbb{C}}, \alpha C)$ must be effective and log canonical, so the pair must satisfy $\alpha \geq \frac{3}{2k}$ due to the requirement that the canonical divisor of the pair $\mathcal{O}_{\mathbb{P}^{2}_{\mathbb{C}}}(-3 + 2k\alpha)$ must be effective,  and $\alpha \leq {\rm min} \bigg\{1, \frac{1}{4}\bigg\}$ for being log-canonical. These observations lead us to $\alpha \in \bigg[\frac{3}{2k}, \frac{1}{4} \bigg]$, and this condition is non-empty provided that $k\geq 6$.

Now we are going to plug the above data into (\ref{logMY}).
We obtain
$$3n \cdot \bigg( 1 - (1-\alpha)^{2}\bigg) + 3t\cdot \bigg(2\alpha + 1 - (1-2\alpha)\bigg) \leq (3\alpha - \alpha^{2})4k^{2} - 6\alpha k,$$
which gives
$$3n(2-\alpha)+12t \leq 4(3-\alpha)k^{2}-6k.$$
Now we are going to use the combinatorial count (\ref{eq:comb}), it leads us to
$$3n(2-\alpha)+12t \leq (3-\alpha)(2n+4t+4k)-6k.$$
Substituting $\alpha=\frac{1}{4}$ in the formula above, we get
$$21n +48t \leq 11(2n+4t+4k)-24k,$$
and after some simple manipulations we land at the inequality
$$t\leq \frac{1}{4}n + 5k.$$
\end{proof}

\section{Nearly-freeness of conic arrangements with the maximal number of nodes and tacnodes}

Here we would like to understand the homological properties of smooth plane conic arrangements with the maximal possible number of tacnodes. In order to do so, we recall some crucial definitions. For a reduced curve $C \subset \mathbb{P}^{2}_{\mathbb{C}}$ of degree $d$ given by $f \in S=\mathbb{C}[x,y,z]$ we denote by $J_{f} = \langle \partial_{x}\,  f, \, \partial_{y} \, f,\partial_{z} \, f \rangle$ the Jacobian ideal and by $\mathfrak{m} = \langle x,y,z \rangle$ the irrelevant ideal. Consider the graded $S$-module $N(f) = I_{f} / J_{f}$, where $I_{f}$ is the saturation of $J_{f}$ with respect to $\mathfrak{m}$.
\begin{definition}
We say that a reduced plane curve $C$ is \emph{nearly free} if $N(f) \neq 0$ and for every $k$ one has ${\rm dim} \, N(f)_{k} \leq 1$. 
\end{definition}
\begin{definition}
We say that a reduced plane curve $C$ is \emph{free} if $N(f) = 0$. 
\end{definition}
Recall that  for a curve $C$ given by $f \in S$ we define the Milnor algebra as $M(f) = S / J_{f}$.
The description of $M(f)$ for nearly free curves comes from \cite{DimcaSticlaru} as follows.
\begin{thm}[Dimca-Sticlaru]
If $C$ is a nearly free curve of degree $d$ given by $f \in S$, then the minimal resolution of the Milnor algebra $M(f)$ has the following form:
\begin{equation*}
\begin{split}
0 \rightarrow S(-b-2(d-1))\rightarrow S(-d_{1}-(d-1))\oplus S(-d_{2}-(d-1)) \oplus S(-d_{3}-(d-1)) \\ \rightarrow S^{3}(-d+1)\rightarrow S \rightarrow M(f) \rightarrow 0
\end{split}
\end{equation*} for some integers $d_{1},d_{2},d_{3}, b$ such that $d_{1} + d_{2} = d$, $d_{2} = d_{3}$, and $b=d_{2}-d+2$. In that case, the pair $(d_{1},d_{2})$ is called the set of exponents of the nearly free curve $C$.
\end{thm}
It is somehow natural to understand the homological properties of curve arrangements with the simplest possible singularities. If we consider a line arrangement $\mathcal{L}$ in $\mathbb{P}^{2}_{\mathbb{C}}$ with $k\geq 3$ lines given by $f=0$ with only double intersection points, then it is known that $\mathcal{L}$ is actually free, i.e., $N(f) = 0$, if and only if $k=3$, so this description looks somehow very restrictive, but it is fully complete. Due to this reason, one tries to relax the notion of the freeness to something less demanding and here the nearly freeness comes into the picture. If we focus on arrangements of smooth plane conics with only nodes and tacnodes as singularities, then such arrangements cannot be free. Now we pass to our proof of Proposition \ref{prop1}.

\begin{proof}
In order to prove that result, let us use (and recall here) a result by du Plessis and Wall in \cite {duPlessisWall}, restated as Corollary 1.2 in \cite{Dimca} by the first author. Let $f = 0$ be the defining equation of $\mathcal{C}$ and we denote by $r$ the minimal degree among all the Jacobian relations, i.e., the minimal degree $r$ for the triple $(a,b,c) \in S_{r}^{3}$ such that $a \cdot \partial_{x}(f) + b \cdot \partial_{y}(f) + c\cdot \partial_{z}(f) = 0$. If $\mathcal{C}$ is free, then
$$(2k-1)^{2} -r(2k-1) + r^2 = \tau(\mathcal{C}),$$
where $\tau (\mathcal{C})$ is the total Tjurina number of $\mathcal{C}$. Since we have only nodes and tacnodes as singularities, then
$$\tau(\CC)=n+3t=2k(k-1)+t,$$
as follows from the combinatorial count (\ref{eq:comb}). If we look at the quadratic equation in $r$ obtained in this way, the corresponding discriminant is
$$\Delta_{F}=4(t-k(k-1))-3 < 0.$$
Indeed, one knows that $t \leq k(k-1)$, and this completes the proof.
\end{proof}
\begin{rk}
Using similar techniques, one can also show that if $\mathcal{C}$ is an arrangement of $k\geq 2$ smooth conics having only nodes as singularities, then $\mathcal{C}$ is never nearly free. Indeed, by \cite[Theorem 1.3]{Dimca}, we need to check whether the equation
$$(2k-1)^{2} - r(2k-1) + r^2 = \tau(\mathcal{C}) + 1 = 2(k^{2}-k) +1,$$
considered with respect to $r$, has integral roots. However, an easy inspection tells us that the arrangement $\mathcal{C}$ is neither nearly free.
\end{rk}
Now we pass to the nearly free case and we are going to prove Proposition \ref{prop2}. In the first step, we show that if $\mathcal{C}$ is nearly free, then $k\leq 4$ and $t=k(k-1)$.
\begin{proof}
Using \cite[Theorem 1.3]{Dimca}, we know that $\CC$ is nearly free if and only if
$$(2k-1)^2-r(2k-1)+r^2-1=\tau(\CC).$$
Since we have only nodes and tacnodes as singularities, we obtain
$$\tau(\CC)=n+3t=2k(k-1)+t,$$
and looking at the quadratic equation in $r$ obtained in this way, we get the discriminant 
$$\Delta_{NF}=4(t-k(k-1))+1.$$
Since $t \leq k(k-1)$, we have that $\Delta_{NF} \geq 0$ if and only if
$t=k(k-1)$. Using now the Miyaoka bound (\ref{miyaoka}) we see that if $\CC$ is nearly free, then
$$k(k-1)=t \leq \frac{4}{9}k^2+\frac{4}{3}k.$$
This inequality implies that $k \leq 4$. The claim about the exponents follows by solving the quadratic equation with respect to $r$.
\end{proof}
Now we are going to show that if $\mathcal{C}$ is an arrangement of $k\leq 4$ conics with $t=k(k-1)$ tacnodes, than $\mathcal{C}$ has to be nearly free. This is going to be done via a detailed discussion on such conic arrangements. Let $\mathcal{C}_{k} = \{C_{1}, ..., C_{k}\} \subset \mathbb{P}^{2}_{\mathbb{C}}$ be an arrangement of $k$ smooth conics admitting only $n$ nodes and $t$ tacnodes with the defining equation $f=f_{1} \cdots f_{k}$, where $C_{i} = V(f_{i})$. We will follow a classification provided by Megyesi in \cite{Meg}. Our computations here are made with an assistance of \verb{Singular{ \cite{Singular}. 

\begin{enumerate}
    \item[0.] If we take a smooth conic $\mathcal{C}_{1}$, then it is well-known that $\mathcal{C}_{1}$ is a nearly free curve having $t=0$ tacnodes.
    \item[1.] Consider $k=2$ and $\mathcal{C}_{2} = \{C_{1},  C_{2}\}$ given by
\begin{equation*}
\begin{array}{l}
    C_{1} : f_{1} = x^2 + y^2 - z^2, \\
    C_{2} : f_{2} = x^{2}/r^{2} + y^2 - z^2,
\end{array}
\end{equation*}
where $r \in \mathbb{C} \setminus \{0,\pm 1\}$. As it was shown by Megyesi, for each $r\in \mathbb{C} \setminus \{0,\pm 1\}$ such an arrangement delivers exactly $2$ tacnodes. Now we prove that $\mathcal{C}_{2}$ is nearly free. Define $Q(x,y,z) = f_{1}\cdot f_{2}$. Observe that ${\rm mdr}(Q)=1$ since we have the following identity 
$$z \cdot \frac{\partial \, Q}{\partial \, y} + y \cdot \frac{\partial \, Q}{\partial \, z} = 0.$$
We are going to apply \cite[Theorem 1.3]{Dimca}, namely
\begin{equation*}
\quad \quad \quad 1-3+9 = r^2 - r(2k-1) + (2k-1)^2 = \tau(\mathcal{C}_{2})+1 = 3\cdot t + 1 = 3\cdot 2 + 1,   
\end{equation*}
so $\mathcal{C}_{2}$ is nearly free.

    \item[2.] Consider $k=3$ and  $\mathcal{C}_{3} = \{C_{1},  C_{2}, C_{3}\}$ given by
\begin{equation*}
\begin{array}{l}
    C_{1} : f_{1} = x^2 + y^2 - z^2, \\
    C_{2} : f_{2} = x^{2}/r^{2} + y^2 - z^2,\\
    C_{3} : f_{3} = x^{2} + y^{2} - r^{2} z^{2}
    \end{array}
\end{equation*}
where $r \in \mathbb{C} \setminus \{0,\pm 1\}$. Again, as it was shown by  Megyesi, for each $r\in \mathbb{C} \setminus \{0,\pm 1\}$ such an arrangement delivers exactly $6$ tacnodes, the maximal possible number of tacnodes for $3$ conics. Using the same tick as for $\mathcal{C}_{2}$, we are going to show that $\mathcal{C}_{3}$ is nearly free. Define $Q(x,y,z) = f_{1}\cdot f_{2} \cdot f_{3}$. Using \verb{Singular{ we can perform computations over $\mathbb{Q}(r)$, where $r$ is the parameter introduced above, and one obtains that ${\rm mdr}(Q) = 3$ -- we skip rather involving computations. We use again \cite[Theorem 1.3]{Dimca}, namely
\begin{equation*}
\quad \quad \quad 9 - 15 + 25 = r^{2}-r(2k-1) + (2k-1)^{2} = \tau(\mathcal{C}_{3})+1 = 3\cdot t + 1 = 19,
\end{equation*}
so $\mathcal{C}_{3}$ is nearly free.
    \item[3.] This example was studied in \cite{MPTG}. We take $\mathcal{C}_{4} = \{C_{1}, C_{2}, C_{3}, C_{4}\}$ such that
\begin{equation*}
\begin{array}{l}
C_{1} : f_{1} = xy-z^{2},\\
C_{2} : f_{2} = xy+z^{2},\\
C_{3} : f_{3} = x^{2} + y^{2} -2z^{2},\\
C_{4} : f_{4} = x^{2} + y^{2} +2z^{2}.
\end{array}
\end{equation*}
It can be checked by a direct computations that $\mathcal{C}_{4}$ has exactly $12$ tacnodes, so for $k=4$ we obtained the maximal possible number of tacnodes according to \ref{eq:naive}. We can compute the minimal free resolution of the Milnor algebra $M(\mathcal{C}_{4})$ which has the following form:
$$0 \rightarrow S(-12) \rightarrow S^{3}(-11) \rightarrow S^{3}(-7) \rightarrow S,$$
${\rm dim} \, N(\mathcal{C}_{4})_{9} = 1$, and ${\rm dim} \,N(\mathcal{C}_{4})_{\ell} = 0$ for $\ell \neq 9$, so our arrangement $\mathcal{C}_{4}$ is nearly free. 
\end{enumerate}
\begin{rk}
Examples given above show that both cases for the exponents as in Proposition \ref{prop2} are possible. Indeed, for $k=2$ conics we get the exponents $(1,3)$, and for $k=3,4$ we have the exponents equal to $(k,k)$.
\end{rk}
\section{Quasi-homogeneous singularities}
Here we want to understand quasi-homogeneous ordinary singularities constructed via pencils of curves (and their slight generalizations). Before we present our main result, let us recall basics related to the quasi-homogeneous singularities.

\begin{definition}
Let $p$ be an  isolated singularity of a polynomial $f\in \mathbb{C}[x,y]. $
We can take local coordinates such that $p=(0,0)$.

The number 
$$\mu_{p}=\dim_\mathbb{C}\left(\mathbb{C}[x,y] /\bigg\langle \frac{\partial f}{\partial x},\frac{\partial f}{\partial y} \bigg\rangle\right)$$
is called the Milnor number of $f$ at $p$.

The number
$$\tau_{p}=\dim_\mathbb{C}\left(\mathbb{C}[ x,y] /\bigg\langle f,\frac{\partial f}{\partial x},\frac{\partial f}{\partial y}\bigg\rangle \right)$$
is called the Tjurina number of $f$ at $p$.
\end{definition}
For a projective situation,
with a point $p\in \mathbb{P}^{2}_{\mathbb{C}}$ and a homogeneous polynomial  $f\in \mathbb{C}[x,y,z]$, we take local affine coordinates such that $p=(0,0,1)$ and then the dehomogenization of $f$.

For a curve $C \, : \, f = 0$ we denote by $\tau(C)$ its total Tjurina number, i.e., 
$$\tau(C) = \sum_{p \in {\rm Sing}(C)} \tau_{p}.$$
\begin{definition}
A singularity is quasi-homogeneous if and only if there exists a holomorphic change of variables so that the defining equation becomes weighted homogeneous, $f(x,y) = \sum_{i,j}c_{i,j}x^{i}y^{j}$ is weighted homogeneous if there exist rational numbers $\alpha, \beta$ such that $\sum_{i,j} c_{i,j}x^{i\cdot \alpha} y^{j \cdot \beta}$ is homogeneous.
\end{definition}
It was proved, for the first time, by Reiffen in \cite{Reiffen} that if $f(x,y)$ is a convergent power series with an isolated singularity at the origin, then $f(x,y)$ is in the ideal generated by the partial derivatives if and only if $f$ is quasi-homogeneous. In particular, it means that in the quasi-homogeneous case one has $\tau_{p} = \mu_{p}$.

Now we are ready to present our result that allows us construct examples of ordinary quasi-homogeneous singularities. We know, based on a private communication with E. Shustin, that every ordinary singularity with multiplicity up to $4$ is quasi-homogeneous -- the argument presented to us follows from \cite{Orev}. Our motivation comes from the following result which is rather well-known -- see for instance \cite{STY, JV}. 
\begin{prop}
\label{prop3}
Let $f_t=g_1+tg_2$ be a pencil of plane curves of degree $k=\deg g_1=\deg g_2$ such that the base locus $B$ consists of $k^2$ points. Then any curve $\CC: f=0$ obtained as the union of $m$ members of this pencils has ordinary quasi-homogeneous singularities of type $(1,1;m)$ at any point in $B$. If all the members which occur in $\CC$ are smooth, then the only singularities of $\CC$ are the points of $B$.
\end{prop}
\begin{rk}
In the setting of the above proposition, there is precise information on $mdr(f)$, see \cite[Theorem 1.8]{D2} for details.
\end{rk}

Note that  the condition $|B|=k^2$ is equivalent to the fact that any point $b \in B$ is smooth on both $D_1:g_1=0$ and $D_2:g_2=0$ and the curves $D_1$ and $D_2$ intersect transversally at $b$. 

Now we are ready to present our generalization of Proposition \ref{prop3}.
 \begin{prop}
\label{prop4}
Let $f_t=g_1+tg_2$ be a pencil of plane curves of degree $k=\deg g_1=\deg g_2$.  For a point  $b \in B$ in the base locus $B$, assume that $D_1$ and $D_2$ are smooth at $b$ and they have a contact of order $c$, i.e., the intersection number $i(D_1,D_2)_{b}$ is equal to $c$. Then any curve $\CC: f=0$ obtained as the union of $m$ reduced members of this pencil has a quasi-homogeneous singularity of type $(c,1;cm)$ at the point $b \in B$. In particular, we have 
$$\mu_{b} = \tau_{b} = (m-1)(cm-1).$$
\end{prop}
\begin{proof}
Since this question is local around $b$, so we can work at the origin of $\C^2$.
The case $c=1$ is clear, so we consider from now on the case $c \geq 2$. Since $D_1$ is smooth at $b=0$, we can choose $u=g_1$ to be a local (analytic) coordinate at the origin. We complete it to a local coordinate system $(u,v)$ at the origin of $\C^2$. In this coordinate system we have $D_1:u=0$ by our construction.
The polynomial $g_2$ in the local coordinates $(u,v)$ has the form
\begin{equation}
\label{E2}
g_2(u,v)=u\cdot a(u,v)+ v^c \cdot b(v),
\end{equation}
where $a(u,v)$ (resp. $b(v)$) is a convergent power series in $(u,v)$ (resp. in $v$) such that $a(0,0) \ne 0$ and $b(0) \ne 0$. This follows from the conditions $i(D_1,D_2)_b=c\geq 2$ and $D_2$ smooth at $b$.
We change now the local coordinates at the origin of $\C^2$ such that  the new coordinates $(\tilde u,\tilde v)$ satisfy the equalities
\begin{equation}
\label{E3}
\tilde u=u\cdot a(u,v) \text{ and } \tilde v^c=v^c \cdot b(v).
\end{equation}
In these new coordinates, we see that
$D_1$ is given by $\tilde u=0$ and $D_2$ is given by $\tilde u+\tilde v^c=0$. It follows that the germ at $b$ of a member $D_1+tD_2$ in the pencil is given by the local equation
\begin{equation}
\label{E4}
g_t=\tilde u+ t(\tilde u+\tilde v^c)=(1+t)\tilde u+t \tilde v^c,
\end{equation}
where $t \ne -1$ since we assume that all the members $C_t$ are reduced curves. Hence the germ of the arrangement $\CC$ at $b$ has a local equation which is a product of $m$ terms of the form 
$$\al \tilde u+\be \tilde v^c,$$
with $\al \ne 0$. If we put the weights ${\rm wt}(\tilde u)=c$ and ${\rm wt}(\tilde v)=1$, we see that all these factors are quasi-homogeneous of type $(c,1;c)$. This completes the proof of our claim, since the formula for the Milnor number of a quasi-homogeneous isolated hypersurface singularity is well-known.
\end{proof}
\begin{ex}
Let us consider the arrangement of conics $\mathcal{C} = \{C_{1}, ..., C_{m}\} \subset \mathbb{P}^{2}_{\mathbb{C}}$, where each equation for $C_t$ can be written as $f_t=g_1 +t g_2$ with
$$g_1=x^2+y^2-yz \text{ and } g_2=(x+y)z.$$
In other words, the curves $C_t$ are members of a pencil of conics, with a base locus $B:g_1=g_2=0$ consisting of $4$ points. By the above results, the curve $\CC_m$, which is the union $C_1 \cup \ldots \cup C_m$, has ordinary quasi-homogeneous singularities. 
\end{ex}
If $p$ is an ordinary singular point of multiplicity  $m$, then it is clear that $\mu_{p} = (m-1)^{2}$, since the corresponding singularity is semi-quasi-homogeneous. To determine the corresponding Tjurina number is much more complicated.
Based on many experiments run with use of \verb{Singular{, if $p$ is an ordinary intersection point of $m$ smooth conics, then we can deduce the following pattern:
$$ \tau_{p} \geq (m-1)^{2} - m + 4.$$
This leads us to the following question.
\begin{question}
Let $p$ be an ordinary singularity of smooth plane conics of multiplicity $m\geq 5$. Then
$$\mu_{p} - \tau_{p} \in \{0, m-4\}.$$
\end{question}

\section*{Acknowledgments}

The first author was partially supported by the Romanian Ministry of Research and Innovation, CNCS - UEFISCDI, Grant \textbf{PN-III-P4-ID-PCE-2020-0029}, within PNCDI III.

The third author was partially supported by the National Science Center (Poland) Sonata Grant Nr \textbf{2018/31/D/ST1/00177}. 

The third author would like to thank Professor Evgenii Shustin and Bernd Schober for fruitful discussions devoted to quasi-homogeneous singularities. The second and the third author would like to thank Marco Golla and Bernd Schober for all comments and remarks that allowed to improve the note. 

Finally, we would like to thank an anonymous referee for all valuable comments and suggestions that allowed to improve the paper.

\end{document}